\documentclass[10pt]{amsart}
\usepackage{amsmath}
\usepackage{amssymb}
\usepackage{amsthm}
\usepackage{amsfonts}
\usepackage{latexsym}
\usepackage{epsfig}
\usepackage{graphicx}
\usepackage{epstopdf}
\newtheorem{theorem}{Theorem}
\newtheorem{lemma}{Lemma}

\newtheorem{remark}{Remark}
\newtheorem{proposition}{Proposition}
\newtheorem{desired result}{Desired result}
\newtheorem{conjecture}{Conjecture}

\newcommand{\begd}{\begin{displaystyle}}
\newcommand{\gl}{\lambda}

\newcommand{\ga}{\alpha}
\newcommand{\gb}{\beta}
\newcommand{\gd}{\delta}
\newcommand{\gD}{\Delta}

\newcommand{\tr}{\textrm{tr}\,}

\newcommand{\ve}{\varepsilon}

\newcommand{\mbq}{\mathbb{Q}}
\newcommand{\mbr}{\mathbb{R}}
\newcommand{\mbz}{\mathbb{Z}}

\newcommand{\mbf}{\mathbb{F}}
\newcommand{\gal}{\textrm{Gal}\,}

\newcommand{\ol}[1]{\overline{#1}}
\newcommand{\mc}[1]{\mathcal{#1}}

\title{A refined version of the Lang-Trotter Conjecture}
\author{Stephan Baier and Nathan Jones}
\date{}

\begin{document}

\maketitle
\begin{abstract}
Let $E$ be an elliptic curve defined over the rational numbers and $r$ a fixed integer.  Using a probabilistic model consistent with the Chebotarev theorem for the division fields of $E$ and the Sato-Tate distribution, Lang and Trotter conjectured an asymptotic formula for the number of primes up to $x$ which have Frobenius trace equal to $r$, where $r$ is a {\it fixed} integer. However, as shown in this note, this asymptotic estimate cannot hold for {\it all} $r$ in the interval $|r|\le 2\sqrt{x}$ with a uniform bound for the error term, because an estimate of this kind would contradict the Chebotarev density theorem as well as the Sato-Tate conjecture.

The purpose of this note is to refine the Lang-Trotter conjecture, by taking into account the "semicircular law", to an asymptotic formula that conjecturally holds for arbitrary integers $r$ in the interval $|r|\le 2\sqrt{x}$, with a uniform error term. We demonstrate consistency of our refinement with the Chebotarev theorem for a fixed division field, and with the Sato-Tate conjecture.  We also present numerical evidence for the refined conjecture.
\end{abstract}

\section{Introduction} \label{introduction}
Let $E$ be an elliptic curve defined over $\mbq$ of minimal discriminant $\gD_E$.  For any prime number $p$ not dividing $\gD_E$, let $E_p$ denote the reduction of $E$ modulo $p$ and
\[
a_E(p) := p + 1 - \# E_p(\mbz/p\mbz)
\]
the trace of Frobenius at $p$.  For a fixed integer $r$, define the prime-counting function
\[
\pi_{E,r}(x):=\sum_{{\begin{substack} {p \leq x, p \nmid \gD_E \\ a_E(p) = r } \end{substack}}} 1.
\]
By studying a probabilistic model consistent with the Chebotarev density theorem for the division fields of $E$ and the Sato-Tate distribution, Lang and Trotter formulated the following conjecture.
\begin{conjecture} \label{ltconjecture}
(Lang-Trotter) Let $E$ be an elliptic curve over $\mbq$ and $r \in \mbz$ a fixed integer.  If $r = 0$ then assume
additionally that $E$ has no complex multiplication.  Then,
\begin{equation} \label{ltasymp}
\pi_{E,r}(x) \quad = \quad C_{E,r} \int_2^x  \frac{dt}{2\sqrt{t}\log t} \; + \; o\left( \frac{\sqrt{x}}{\log x} \right) = C_{E,r} \cdot \frac{\sqrt{x}}{\log x}+
o\left( \frac{\sqrt{x}}{\log x} \right)
\end{equation}
as $x \rightarrow \infty$, where $C_{E,r}$ is a specific non-negative constant.
\end{conjecture}

\begin{remark}
It is possible that the constant $C_{E,r} = 0$, in which case we interpret the asymptotic to mean that there are only finitely many primes $p$ for which $a_E(p) = r$.
\end{remark}

We note that if $r=0$ and $E$ has complex multiplication, Deuring \cite{deur} showed that half of the primes $p$ satisfy $a_E(p)=0$, {\it i.e.},
$$
\pi_{E,0}(x)\sim \frac{\pi(x)}{2} \ \ \ \ \mbox{ as } x\rightarrow\infty.
$$
More precisely, for any constant $C>1$, we have
\begin{equation} \label{deuob}
\pi_{E,0}(x)= \frac{1}{2} Li(x)+O\left(\frac{x}{(\log x)^C}\right),
\end{equation}
where the implied $O$-constant depends only on $E$ and $C$, and
\[
Li(x) \; := \; \int_2^x \frac{dt}{\log t} \; \sim \; \frac{x}{\log x} \quad \text{ as } x \rightarrow \infty.
\]
Primes $p$ with $a_E(p)=0$ are known as ``supersingular primes''.

We point out that Conjecture 1 is formulated for {\it fixed} numbers $r$. The purpose of this note is to refine Conjecture \ref{ltconjecture} to an asymptotic formula which conjecturally holds for {\it arbitrary} integers $r$ in the interval $-2\sqrt{x}\le r\le 2\sqrt{x}$, with a uniform error term, where the case $r=0$ is excluded if $E$ has complex multiplication. Our refinement is stated below.

\begin{conjecture} \label{refinement}
Let $E$ be an elliptic curve over $\mbq$. Fix any $C>1$. Then, uniformly for $|r| \leq 2\sqrt{x}$, where the case $r=0$ is excluded if $E$ has CM, we have
\begin{equation} \label{refconj}
\pi_{E,r}(x) \; = \;  C_{E,r}\int_{\max\{2,r^2/4\}}^x \frac{\Phi_E(r/(2\sqrt{t}))}{2\sqrt{t} \log t} dt+ O_{E,C}\left(\frac{\sqrt{x}}{(\log x)^C}\right),
\end{equation}
where $C_{E,r}$ is the same constant appearing in Conjecture 1, and
\begin{equation} \label{phidef2}
\Phi_E(z):=\left\{\begin{array}{llll} \sqrt{1-z^2} & \mbox{\rm if } E\ \mbox{\rm does not have CM} \\ \\
\frac{1}{\sqrt{1-z^2}} & \mbox{\rm if } E\ \mbox{\rm has CM.}
\end{array}\right.
\end{equation}
\end{conjecture}

For convenience, throughout the sequel, we denote the main term on the right-hand side of \eqref{refconj} by $F_{E,r}(x)$, {\it i.e.}, we set
\begin{equation} \label{fdef}
F_{E,r}(x) := C_{E,r} \int_{\max\{2,r^2/4\}}^x \frac{\Phi_E(r/(2\sqrt{t}))}{2\sqrt{t} \log t} dt
\end{equation}
if $x\ge \max\{2,r^2/4\}$. We note that this term is bounded from above by the main term in Conjecture 1, {\it i.e.}
\begin{equation} \label{Fbound}
F_{E,r}(x)\ll C_{E,r}\cdot \frac{\sqrt{x}}{\log x}.
\end{equation}

Conjecture 2 is rather ``conservative'' in the sense that the $O$-term bounding the error is smaller than the main term by no more than a factor of a power of logarithm.
In section 3, we shall give a heuristic suggesting the following sharpening of Conjecture 2 which essentially states that the the error term in \eqref{refconj} should not be much larger than the square root of the main term.

\begin{conjecture} \label{refinement2}
Let $E$ be an elliptic curve over $\mbq$ and $\ve > 0$. Assume that $|r| \leq 2\sqrt{x}$. Assume further that $r\not= 0$ if $E$ has CM. Then
\begin{equation} \label{secondcon}
\pi_{E,r}(x) = F_{E,r}(x)+ O_{E,\ve}\left( x^\ve \sqrt{1 + F_{E,r}(x)}\right),
\end{equation}
where the function $F_{E,r}(x)$ is defined as in \eqref{fdef}.
\end{conjecture}

Our work is motivated by the natural desire to sum the prime-counting function $\pi_{E,r}(x)$ over $r$ in a fixed residue class and recover the Chebotarev density theorem for the appropriate division field of $E$.
More precisely, fix a modulus $q$ and denote by $\mbq(E[q])$
the $q$-th division field of $E$, i.e. the field obtained by adjoining to $\mbq$ the $x$ and $y$ coordinates of the $q$-torsion points of a given Weierstrass model of $E$.  Fixing a basis
\[
E[q] \simeq \mbz/q\mbz \oplus \mbz/q\mbz
\]
of $E[q]$ over $\mbz/q\mbz$, we may view the Galois group
\[
\gal(\mbq(E[q])/\mbq) \leq GL_2(\mbz/q\mbz)
\]
as a subgroup of $GL_2(\mbz/q\mbz)$.  Finally, let us denote by
\begin{equation*} \label{chebfactordef}
\gd_{a,q} := \frac{ | \{ g \in \gal(\mbq(E[q])/\mbq) : \tr g \equiv a \mod q \} | }{ | \gal(\mbq(E[q])/\mbq) | }
\end{equation*}
the Chebotarev factor.   The Chebotarev density theorem for the field $\mbq(E[q])$ implies that, for any fixed constant $C>1$, we have
\begin{equation} \label{chebotarev}
\sum_{{\begin{substack} {p \leq x, p \nmid q \gD_E \\ a_E(p) \equiv a \mod q } \end{substack}}} 1 \quad = \quad \gd_{a,q} \cdot Li(x) \; + \; O\left( \frac{x}{(\log x)^C} \right).
\end{equation}
We begin by observing that
\[
\sum_{{\begin{substack} {r \equiv a \mod q \\ |r| \leq 2\sqrt{x}} \end{substack}}} \left( \sum_{{\begin{substack} {p \leq x, p \nmid \gD_E \\ a_E(p) = r } \end{substack}}} 1 \right) \quad = \quad \left( \sum_{{\begin{substack} {p \leq x, p \nmid q \gD_E \\ a_E(p) \equiv a \mod q } \end{substack}}} 1 \right) \; + \; O_q(1).
\]
Thus, paying attention only to the main terms in \eqref{chebotarev} and Conjecture \ref{ltconjecture}, and taking \eqref{deuob} in the CM case into account, it is natural to hope that
\[
\left( \sum_{{\begin{substack} {r \equiv a \mod q \\ 0<|r| \leq 2\sqrt{x}} \end{substack}}} C_{E,r} \right) \cdot \frac{\sqrt{x}}{\log x} \quad \sim \quad \left(\gd_{a,q}-\frac{\gamma({E,a,q})}{2}\right) \cdot \frac{x}{\log x},
\]
where
\begin{equation*} \label{epsdef}
\gamma({E,a,q})=\left\{ \begin{array}{llll} 1 & \mbox{ if } E \mbox{ has CM and } a\equiv 0\mod q,\\ \\ 0 & \mbox{ otherwise.}\end{array} \right.
\end{equation*}
However, this is not the case.  In fact, as proved in Section \ref{constantaveraging}, one has the following.
\begin{proposition} \label{asymprop}
Let $A$ be any integer and $B$ any positive integer.  Set $M :=  \max \{ |A|, |A+B| \}$.  Then
\begin{equation} \label{asymptotic}
\sum_{{\begin{substack} {r \equiv a \mod q \\ A < r \leq A+B\\ r\not=0} \end{substack}}} C_{E,r} =
\begin{cases}
\frac{1}{\pi} \cdot \left(\gd_{a,q} - \frac{\gamma({E,a,q})}{2}\right) \cdot B + O_{E} \left( q \cdot \log^{3} M \right) & \text{ if $E$ has CM} \\
\frac{2}{\pi} \cdot \gd_{a,q} \cdot B + O_E(q) & \text{ if $E$ has no CM.}
\end{cases}
\end{equation}
\end{proposition}
It follows that
\begin{equation} \label{discrepancy}
\left( \sum_{{\begin{substack} {r \equiv a \mod q \\ 0<|r| \leq 2\sqrt{x}} \end{substack}}} C_{E,r} \right) \cdot \frac{\sqrt{x}}{\log x} \quad \sim \quad \frac{1}{\gl}_E\cdot \frac{8}{\pi} \cdot \left(\gd_{a,q} - \frac{\gamma({E,a,q})}{2}\right) \cdot \frac{x}{\log x}
\end{equation}
as $x \rightarrow \infty$, where
\[
\gl_E :=
\begin{cases}
	2 & \text{ if $E$ has CM} \\
	1 & \text{ if $E$ has no CM.}
\end{cases}
\]

Hence, the conjectural asymptotic estimate \eqref{ltasymp} cannot hold for all $r$ in the interval $|r|\le 2\sqrt{x}$ with a uniform bound for the error term of size $o(\sqrt{x}/\log x)$.
We shall further see that \eqref{ltasymp} with a uniform bound for the error term also contradicts the Sato-Tate conjecture (this follows from Theorem \ref{satotateth}, proved in section \ref{consistencywithsatotate}). We shall show that our refined  Conjecture 2 (resp. Conjecture 3) remedies these discrepancies. Moreover, in section 7 we shall demonstrate that in a certain sense, the main term in Conjectures 2 (resp. Conjecture 3) is the only possibility.

The paper is organized as follows.
In Section \ref{therefinement}, we motivate the Conjectures 2 and 3.
We also discuss briefly in which regions of the $(x,r)$-plane our main term differs significantly from that in Conjecture \ref{ltconjecture} and under which circumstances \eqref{refconj} (resp. \eqref{secondcon}) is actually an {\it asymptotic} estimate. In Section \ref{constantsexplicit}, we give a detailed description of the constants $C_{E,r}$. In Section \ref{constantaveraging}, we provide a proof of Proposition \ref{asymprop} which will serve as a key tool in what follows. In Section \ref{consistencywithchebotarevdensity}, we prove that Conjecture \ref{refinement} is consistent with the Chebotarev density theorem, and
in Section \ref{consistencywithsatotate}, we demonstrate the consistency with the distribution of $a_E(p)/(2\sqrt{p}) \in [-1,1]$.  Finally, in Section \ref{numericalevidence}, we present numerical evidence for Conjecture \ref{refinement}.

\section{Acknowledgments} \label{acknowledgments}
We would like to thank A. Granville for helpful comments on an earlier version and J. Fearnley for advice regarding the numerical computations. Moreover, we wish to thank the referee for many valuable comments.

\section{The refinement} \label{therefinement}
The work of Lang-Trotter takes account of algebraic and analytic information in coming up with the factor $C_{E,r}$. However, the analytic part of their heuristic replaces the ``semi-circular law'' of Sato-Tate with a limiting constant value. This works well for fixed (or small) $r$'s, as considered in their work.  However, when we consider arbitrary $r$'s in the interval $-2\sqrt{x}\le r\le 2\sqrt{x}$, it becomes necessary to introduce an analytic factor corresponding to the Sato-Tate distribution in the non-CM case and to another characteristic distribution in the CM case.

Roughly speaking, the heuristics of Lang and Trotter predict that the probability that a large natural number $p$ is prime and satisfies
$a_E(p) = r$ is
\begin{equation} \label{ltheuristic}
\approx C_{E,r} \cdot \frac{1}{2\sqrt{p}\log p}.
\end{equation}
Thus, one expects that
\[
\pi_{E,r}(x)=C_{E,r}\cdot \sum_{2\leq n \leq x} \frac{1}{2\sqrt{n}\log n} + o\left(\frac{\sqrt{x}}{\log x}\right)=C_{E,r} \int_2^x  \frac{dt}{2\sqrt{t}\log t}+ o\left(\frac{\sqrt{x}}{\log x}\right),
\]
as $x\rightarrow \infty$. We note that
\[
\int_2^x  \frac{dt}{2\sqrt{t}\log t} \sim \frac{\sqrt{x}}{\log x}, \quad \mbox{as } x\rightarrow \infty.
\]

To be precise, the reason for the apparent inconsistency \eqref{discrepancy} between Conjecture \ref{ltconjecture} and \eqref{chebotarev} is twofold:
\begin{enumerate}
\item[R1]{When $r$ is not very small compared with $\sqrt{x}$, the heuristic \eqref{ltheuristic} needs to be corrected by a factor accounting for the distribution of
\[
\frac{a_E(p)}{2\sqrt{p}} \in [-1,1].
\]
}
\item[R2]{since $\pi_{E,r}(x)$ only counts primes $p$ which are $\geq r^2/4$, the interval of integration in Conjecture \ref{ltconjecture} should be $[r^2/4,x]$ rather than $[2,x]$.}
\end{enumerate}
Note that for fixed $r$ and large $x$, neither of these observations affect the asymptotic.

\subsection{The distribution of $a_E(p)/(2\sqrt{p}) \in [-1,1]$}

The appropriate measure for equidistribution of the quantity
\[
\frac{a_E(p)}{2\sqrt{p}} \in [-1,1]
\]
is $\phi_E(z)dz$, where $\phi_E(z)$ is defined by
\begin{equation} \label{phidef}
\phi_E(z):=\left\{\begin{array}{llll} \frac{2}{\pi} \sqrt{1-z^2} & \mbox{\rm if } E\ \mbox{\rm does not have CM} \\ \\
\frac{1}{2 \pi} \cdot \frac{1}{\sqrt{1-z^2}} & \mbox{\rm if } E\ \mbox{\rm has CM.}
\end{array}\right.
\end{equation}
In the CM case, this distribution law is a classical theorem of Deuring \cite{deur}.

\begin{theorem} (Deuring) \label{deuringsthm}
Suppose that $K$ is an imaginary quadratic field and that $E$ has complex multiplication by an order in $K$, i.e. that
\[
\text{End}_{\ol{\mbq}}(E) \otimes \mbq \simeq K.
\]
Then for any prime number $p$ of good reduction for $E$, we have
\[
a_E(p) = 0 \; \Longleftrightarrow \; p \text{ is inert in } K.
\]
Furthermore, if $I \subset [-1,1]$ is some interval with $0 \notin I$, then
\begin{equation} \label{cmdistr}
\lim_{x \rightarrow \infty} \frac{| \{ p \leq x : p \nmid \gD_E, \; \frac{a_E(p)}{2\sqrt{p}} \in I \} | }{\pi(x)} = \int_I \phi_E(z)dz,
\end{equation}
where
\[
\phi_E(z) = \frac{1}{2 \pi} \cdot \frac{1}{\sqrt{1-z^2}}.
\]
\end{theorem}

In the non-CM case, the distribution law was conjectured independently by Sato and Tate (see \cite{Tate}).

\begin{conjecture} \label{stconjecture}  (Sato-Tate)
For an elliptic curve $E$ over $\mbq$ without complex multiplication and any subinterval $I \subseteq [-1, 1]$, we have
\begin{equation*} \label{satotate}
\lim_{x \rightarrow \infty} \frac{ | \{ p \leq x : p \nmid \gD_E, \frac{a_E(p)}{2\sqrt{p}} \in I \} | }{\pi(x)} \sim \int_I \phi_E(z) dz,
\end{equation*}
where $\phi_E(z)=\frac{2}{\pi} \sqrt{1-z^2}$.
\end{conjecture}
We note that the Sato-Tate conjecture has been proved by L. Clozel, M. Harris, N. Shepherd-Barron and R. Taylor for all elliptic curves $E$ over totally real fields (in particular, over the rationals) satisfying the mild condition of having multiplicative reduction at some prime (see \cite{Tayl} and the references therein).

\subsection{Modifying the heuristic}

Observation R1 and the above facts on the distribution of $a_E(p)/(2\sqrt{p})$ suggest that the heuristic \eqref{ltheuristic} should be corrected by the factor $\phi_E(r/(2\sqrt{p}))$ and then be normalized by an appropriate constant factor $\mathfrak{C}$ which we specify later.  Hence, the probability that a large natural number $p$ is a prime with $a_E(p) = r$ should be
\[
\sim \mathfrak{C} \cdot C_{E,r}\frac{\phi_E(r/(2\sqrt{p}))}{2\sqrt{p} \log p}.
\]
This modified heuristic, taken together with observation R2, suggests that the prime counting function $\pi_{E,r}(x)$ behaves approximately like
$$
\mathfrak{C} \cdot C_{E,r}\int_{\max\{2,r^2/4\}}^x \frac{\phi_E(r/(2\sqrt{t}))}{2\sqrt{t} \log t} dt.
$$
For this approximation to be consistent with Conjecture 1, the normalization factor must be $\mathfrak{C}=1/\phi_E(0)$. This leads us to the main term in Conjecture 2 upon noting that $\Phi_E(z)=\phi_E(z)/\phi_E(0)$.

Furthermore, taking the bound \eqref{Fbound} and the order of magnitude of the $O$-term in \eqref{chebotarev} into account, it seems reasonable to conjecture that the error in our approximation of $\pi_{E,r}(x)$ is smaller than $\sqrt{x}/\log x$ by at least a factor of $(\log x)^C$, which gives the error term in Conjecture 2. 

We are not only interested in the correct form of the main term in the approximation of $\pi_{E,r}(x)$ but also in the true order of magnitude of the error term. We note that, assuming the generalized Riemann hypothesis, the true order of magnitude of the error term in \eqref{chebotarev} is $O\left(x^{1/2+\varepsilon}\right)$ which is by almost a factor of $\sqrt{x}$ smaller than $x/(\log x)^C$. Similarly,
a much sharper bound than $O(\sqrt{x}/(\log x)^C)$ should hold for the error term in \eqref{ltasymp}. Indeed, if we assume that the events ``$p$ is prime with $a_E(p) = r$'' are independent as $p$ runs over the natural numbers, then Chebyshev's law of large numbers suggests that the error term should not be much larger than the square root of the main term. This leads us to the error term
\begin{equation} \label{errorterm}
O_{E,\varepsilon}\left(x^\ve \sqrt{1 + F_{E,r}(x)}\right)
\end{equation}
in Conjecture 3.
The reason we include the term $1$ in the error bound is so that the statement continues to hold true even when $C_{E,r} = 0$.

We point out that the implied constant in \eqref{errorterm} cannot be independent of the elliptic curve $E$. To see this, pick any large integer $r$ and
then for any prime $p$ in the range $r^2 /4 < p < x$, find $E_p$ (mod $p$) such that $a_p(E_p ) = r$. Then select an elliptic curve $E$ over $\mbq$ with $E \equiv E_p$ (mod $p$) for all such $p$. Hence, if one desires to state a conjecture which is uniform in $E$, one certainly needs to bring into the error term some information about $E$. It is conceivable that one might replace \eqref{errorterm} with
\[
O_{\ve}\left( (N_E \cdot x)^\ve \sqrt{1 +F_{E,r}(x)}\right),
\]
where $N_E$ is the conductor of $E$.

\subsection{Comparison of the main term and error term} \label{maintermanderrorterm}
In the following, we compare the sizes of the main term $F_{E,r}(x)$ and the error terms in \eqref{refconj} and \eqref{secondcon}, respectively.  We begin with some comments on the constants (see Section \ref{constantsexplicit} for details).

If $C_{E,r}=0$, then it is conjectured that only finitely many primes satisfy $a_E(p)=r$ in which case \eqref{refconj} is trivial. In the following, we assume that $C_{E,r}\not=0$.
If $E$ has CM by an order in an imaginary quadratic field $K$, then (as we will see in Section \ref{constantsexplicit}) the nonzero values of the constant satisfy the bound
\begin{equation} \label{CMcasebounds}
C_{E,r} \neq 0 \; \Longrightarrow \; \forall \, r, \; \frac{1}{\log \log (3+|r|)} \ll_E C_{E,r} \ll_E \log \log (3+|r|).
\end{equation}
It follows that in this case the main term $F_{E,r}(x)$ satisfies the bound
$$
\frac{\sqrt{4x-r^2}}{\log x \cdot \log\log (3+|r|)}\ll_E F_{E,r}(x)\ll_E
\frac{\sqrt{4x-r^2}}{\log x} \cdot \log\log (3+|r|).
$$
If $E$ has no CM, then the nonzero values of the constant $C_{E,r}$ are uniformly bounded from below and above as $r$ varies, i.e. there exist positive constants $c_E$ and $C_E$ for which
\begin{equation} \label{nonCMcasebounds}
C_{E,r} \neq 0 \; \Longrightarrow \; \forall \, r, \; c_E \leq C_{E,r} \leq C_E.
\end{equation}
Hence, in this case the main term in $F_{E,r}(x)$ satisfies the bound
$$
F_{E,r}(x)\asymp_E \frac{(4x-r^2)^{3/2}}{x\log x}.
$$

Now, let $B>0$ be arbitrarily given. By the above observations, if the constant $C$ is chosen large enough, then the error term $O\left(\sqrt{x}/(\log x)^C\right)$ in \eqref{refconj} is small compared to the main term if $|r|\le 2\sqrt{x}\left(1-1/(\log x)^B\right)$ and $x$ is sufficiently large. Hence, in this case, \eqref{refconj} is an asymptotic estimate. In all other cases, \eqref{refconj} implies an estimate for $\pi_{E,r}(x)$ which is still non-trivial.

Similarly, \eqref{secondcon} is an asymptotic estimate if $x$ is large and $|r|\le 2\sqrt{x}\left(1-x^{-\delta}\right)$, where $\delta$ is a fixed positive number satisfying $\delta<1$ in the CM case and $\delta<1/3$ in the non-CM case (provided $\varepsilon$ is chosen small enough).

We also note that \eqref{refconj} as well as \eqref{secondcon} imply \eqref{ltasymp} if
$r=o(\sqrt{x})$. Hence, Conjecture 1 is contained in Conjecture 2 as well as in Conjecture 4.
If $|r|>D\sqrt{x}$ for some fixed positive $D$, then the main term in \eqref{ltasymp} is significantly larger than the main term $F_{E,r}(x)$ in \eqref{refconj} and \eqref{secondcon}.

\section{The constants $C_{E,r}$} \label{constantsexplicit}
We now give a description of the constants $C_{E,r}$. The reader may find more details in \cite{langtrotter}. 

We first introduce the notation
\begin{equation*} \label{defofGEN}
G_E(n) :=
\begin{cases}
\gal(K(E[n])/K)  & \text{ if $E$ has CM by the imaginary quadratic field $K$} \\
\gal(\mbq(E[n])/\mbq) & \text{ if $E$ has no CM.}
\end{cases}
\end{equation*}
We further set
\[ H_E(n) :=
\begin{cases} (\mc{O}/n\mc{O})^*
& \text{ if $E$ has CM by an order $\mc{O}$ in $K$} \\
GL_2(\mbz/n\mbz) & \text{ if $E$ has no CM.} 
\end{cases}
\]
We note that $G_E(n)$ can be viewed as a subgroup of $H_E(n)$. With this in mind, for any subgroup $G$ of $H_E(n)$ and any integer $r$, we write
\[
G_r := \{ g \in G \, : \, \tr g \equiv r \mod n \}.
\]
Following Lang and Trotter \cite{langtrotter}, we now define the constant $C_{E,r}$ by 
\begin{equation} \label{cdef}
C_{E,r} :=  \phi_E(0) \cdot \frac{m_E |G_E(m_E)_r|}{|G_E(m_E)|} \cdot \prod_{{\begin{substack} {\ell \text{ prime} \\ \ell \nmid m_E} \end{substack}}} \frac{\ell |H_E(\ell)_r|}{|H_E(\ell)|},
\end{equation} 
where the positive integer $m_E$ is given by the following theorem, the celebrated non-CM case of which is due to Serre \cite{serre}.

\begin{theorem} \label{serrestheorem}
Suppose that $E$ is an elliptic curve over $\mbq$.  Then there exists a positive integer $m_E$ so that, for any positive integer $n$, we have
\[
G_E(n) \simeq \pi^{-1}(G_E(\gcd(n,m_E))),
\]
where $\pi : H_E(n) \longrightarrow H_E(\gcd(n,m_E))$ denotes the canonical projection. In particular, if $\ell$ is a prime not dividing $m_E$, then $G_E(\ell)\simeq H_E(\ell)$.
\end{theorem}

Note that the conclusion of Theorem \ref{serrestheorem} and \eqref{cdef} continue to hold if one replaces the integer $m_E$ by any multiple.  For notational convenience, we will assume in the CM case that
\begin{equation} \label{meassum}
\left( 4 \cdot \prod_{\ell \text{ ramified in } \mc{O}} \ell \right) \quad \text{ divides } \;  m_E.
\end{equation}
Under this assumption, we further have the following explicit description of the cardinalities of $H_E(\ell)$ and $H_E(\ell)_r$ if $\ell \nmid m_E$.

\begin{lemma} \label{Hcard} Let $r$ be any integer and $\ell$ be a prime not dividing $m_E$ (in particular, $\ell$ does not ramify in $\mc{O}$ if $E$ has CM by $\mc{O}$). If $E$ does not have CM, then
\begin{equation} \label{H1} 
|H_E(\ell)| = l(l-1)^2(l+1)
\end{equation}  
and
\begin{equation} \label{H2} 
|H_E(\ell)_r | =
\begin{cases}
	\ell^2(\ell - 1) & \text{ if } r \equiv 0 \mod \ell \\
	\ell(\ell^2 - \ell - 1) & \text{ otherwise.}
\end{cases}
\end{equation} 
If $E$ has CM by an order $\mc{O}$ in an imaginary quadratic field $K$, then
\begin{equation} \label{H3}
|H_E(\ell)| = (\ell - 1)(\ell - \chi_\mc{O}(\ell))
\end{equation}  
and 
\begin{equation} \label{H4} 
| H_E(\ell)_r | = 
\begin{cases}
	\ell - 1 & \text{ if $r \equiv 0 \mod \ell$} \\
	\ell - (1 + \chi_\mc{O}(\ell)) & \text{ otherwise,}
\end{cases}
\end{equation}
where $\chi_\mc{O}(\ell)$ is the character determining the splitting of $\ell$ in the order $\mc{O}$, namely
\[
\chi_\mc{O}(\ell) :=
	\begin{cases}
		1 & \text{ if $\ell$ splits in $\mc{O}$} \\
		-1 & \text{ if $\ell$ is inert in $\mc{O}$.} 
	\end{cases}
\]
\end{lemma}

\begin{proof} We leave the proofs of \eqref{H1} and \eqref{H2} to the reader and deal only with the CM case. We note that since $E$ is defined over $\mbq$, the class number $h(\mc{O})$ of the order $\mc{O}$ equals 1
(see \cite[p. $99$, Proposition$1.2$ (b)]{silverman2}, which works out the case where $\mc{O}$ is the full ring of integers).  Now for any prime $\ell$ which is not ramified in $\mc{O}$, we have
\begin{equation} \label{OmodellO}
\mc{O} / \ell \mc{O} \simeq
\begin{cases}
	\mbf_\ell \oplus \mbf_\ell & \text{ if $\ell$ splits in $\mc{O}$} \\
	\mbf_{\ell^2} & \text{ if $\ell$ is inert in $\mc{O}$.}
\end{cases}
\end{equation}
This implies \eqref{H3} and \eqref{H4}.
\end{proof}

Thus, if $E$ has CM, then 
\begin{equation*} \label{explicitconstantCMcase}
C_{E,r}= \frac{1}{2\pi} \cdot \frac{m_E |G_E(m_E)_r|}{|G_E(m_E)|}\cdot \prod_{{\begin{substack} {\ell \nmid m_E \\ \ell \mid r} \end{substack}}} \left(1+\frac{\chi_\mc{O}(\ell)}{\ell - \chi_\mc{O}(\ell)}\right) \prod_{{\begin{substack} {\ell \nmid m_E \\ \ell \nmid r} \end{substack}}} \left(1-\frac{\chi_\mc{O}(\ell)}{\left( \ell - 1 \right) \left( \ell - \chi_\mc{O}(\ell) \right)}\right).
\end{equation*}
Noting that, for fixed $E$, the factor $\frac{m_E |G_E(m_E)_r|}{|G_E(m_E)|}$ takes on only finitely many values as $r$ varies, we obtain the bounds \eqref{CMcasebounds}. If $E$ does not have CM, then we have 
\begin{equation*} \label{explicitconstant}
C_{E,r} = \frac{2}{\pi} \cdot \frac{m_E |G_E(m_E)_r|}{|G_E(m_E)|}\cdot \prod_{{\begin{substack} {\ell \nmid m_E\\ \ell \mid r} \end{substack}}} \left( 1 + \frac{1}{\ell^2-1} \right) \cdot \prod_{{\begin{substack} {\ell \nmid m_E\\ \ell \mid r} \end{substack}}} \left( 1 -  \frac{1}{(\ell-1)(\ell^2-1)} \right).
\end{equation*}
Notice that the Euler product converges absolutely and \eqref{nonCMcasebounds} follows.

\section{Averaging the constants over residue classes} \label{constantaveraging}
Proposition \ref{asymprop} will serve as a key tool in the following sections. We now provide a proof of this proposition which has some similarity to the proof of Lemma 9 in \cite{baierzhao}, where certain constants related to $C_{E,r}$ were averaged as well. However, the algebraic structure of the relevant constants in \cite{baierzhao} is much simpler than that of the original Lang-Trotter constants $C_{E,r}$ which we deal with in the present paper.

Our goal is to obtain an asymptotic formula for the average
\[ \sum_{{\begin{substack} {r \equiv a \mod q \\ A < r \leq A+B \\ r \neq 0} \end{substack}}} C_{E,r}. \]
As observed in the previous section, \eqref{cdef} continues to hold if the integer $m_E$ appearing on the right-hand side of this equation is replaced by any multiple. Therefore, we may assume that $m_E$ is divisible by $q$ throughout the following. Moreover,  Theorem \ref{serrestheorem} and Lemma \ref{Hcard} tell us that if $\ell$ is a prime not dividing $m_E$, then 
\[|G_E(\ell)_r|=\begin{cases} |H_E(\ell)_1| & \mbox{ if } \ell \nmid r \\
|H_E(\ell)_0| & \mbox{ if } \ell | r. \end{cases} \] 
Hence, we can write the constant in the form
\begin{equation} \label{usefulcer}
C_{E,r} = \phi_E(0) \cdot \frac{m_E |G_E(m_E)_r|}{|G_E(m_E)|} \cdot C \cdot f(r),
\end{equation}
where
\[
C := \prod_{{\begin{substack} { \ell \nmid m_E } \end{substack}}} \frac{\ell |H_E(\ell)_1|}{|H_E(\ell)|} \quad \text{ and } \quad f(r) := \prod_{{\begin{substack} { \ell \nmid m_E \\ \ell \mid r } \end{substack}}} \frac{|H_E(\ell)_0|}{|H_E(\ell)_1|}.
\]
Thus we have
\begin{eqnarray}
\sum_{{\begin{substack} {r \equiv a \mod q \\ A < r \leq A+B \\ r \neq 0} \end{substack}}} C_{E,r}  &=& \phi_E(0) \cdot m_E \cdot
C \cdot \sum_{{\begin{substack} {r \equiv a \mod q \\ A < r \leq A+B \\ r \neq 0} \end{substack}}} f(r) \cdot \frac{|G_E(m_E)_r|}{|G_E(m_E)|}
\nonumber\\ \label{finalsum1}
&=& \phi_E(0) \cdot m_E \cdot
C \cdot \sum_{{\begin{substack} {b \mod m_E \\ b \equiv a \mod q} \end{substack}}} \frac{|G_E(m_E)_b|}{|G_E(m_E)|} \sum_{{\begin{substack} {r \equiv b \mod m_E \\ A < r \leq A+B \\ r \neq 0} \end{substack}}} f(r).
\end{eqnarray}
We use the Dirichlet convolution
$g = f * \mu$ to re-write the inner sum as
\begin{equation} \label{mobiussum}
\sum_{{\begin{substack} {r \equiv b \mod m_E \\ A < r \leq A+B \\ r \neq 0} \end{substack}}} f(r) =  \sum_{{\begin{substack} {r \equiv b \mod m_E \\ A < r \leq A+B \\ r \neq 0} \end{substack}}}
\sum_{d \mid r} g(d)
= \sum_{d = 1}^\infty g(d) \sum_{{\begin{substack} {A < r \leq A+B\\
r \equiv b \mod m_E \\ r \equiv 0 \mod d \\ r \neq 0} \end{substack}}} 1.
\end{equation}
It is straightforward to show that
\begin{equation} \label{g}
g(d) =
\begin{cases}
\mu^2(d) \cdot \prod_{\ell \mid d} \frac{|H_E(\ell)_0| - |H_E(\ell)_1|}{ |H_E(\ell)_1|} & \text{ if } \gcd(d,m_E) = 1 \\
0 & \text{ if } \gcd(d,m_E) > 1.
\end{cases}
\end{equation}
Using Lemma \ref{Hcard}, we see that
\[
\sum_{d = 1}^\infty g(d) =
\begin{cases}
	\sum_{\gcd(d,m_E) = 1} \frac{\mu^2(d) \cdot \chi_\mc{O}(d)}{\prod_{\ell \mid d} \left( \ell - 1 - \chi_\mc{O}(\ell) \right)} & \text{ if $E$ has CM by $\mc{O}$} \\
	\sum_{\gcd(d,m_E) = 1} \frac{\mu^2(d)}{\prod_{\ell \mid d} (\ell^2 - \ell - 1)} & \text{ if $E$ has no CM.}
\end{cases}
\]
Note in particular that the sum $\sum_{d = 1}^\infty g(d)$ is convergent, albeit only conditionally in the CM case.  Using \eqref{mobiussum}, \eqref{g} and the Chinese Remainder Theorem, we deduce in the non-CM case that
\[
\sum_{{\begin{substack} {r \equiv b \mod m_E \\ A < r \leq A+B \\ r \neq 0} \end{substack}}} f(r) =
\sum_{d = 1}^\infty g(d) \left(\frac{B}{dm_E} + O(1) \right) = \frac{B}{m_E} \cdot \sum_{d=1}^\infty \frac{g(d)}{d} + O(1).
\]
In the CM case, the error term is more delicate, and so (recalling that $M := \max\{ |A|, |A+B| \}$), we write
\begin{eqnarray*}
\sum_{{\begin{substack} {r \equiv b \mod m_E \\ A < r \leq A+B \\ r \neq 0} \end{substack}}} f(r) &=&
\sum_{d \leq M} g(d) \left(\frac{B}{dm_E} + O(1) \right) \\
&=& \frac{B}{m_E} \cdot \sum_{d = 1}^\infty \frac{g(d)}{d} + O\left( \frac{B}{m_E} \cdot \sum_{d > M} \frac{|g(d)|}{d} + \sum_{d \leq M} |g(d)| \right) \\
&=&  \frac{B}{m_E} \cdot \sum_{d = 1}^\infty \frac{g(d)}{d} + O\left( (\log M)^{3} \right),
\end{eqnarray*}
where we have used the fact that, for $\gcd(d,m_E) = 1$ (note that then $d$ must be odd by \eqref{meassum}), one has
\[
|g(d)| \leq \frac{1}{d} \cdot \prod_{\ell \mid d} \left( 1 + \frac{2}{\ell - 2} \right) \ll \frac{1}{d} \cdot \prod_{\ell \mid d} \left( 1 + \frac{2}{\ell} \right) \ll \frac{(\log d)^{2}}{d} \; \text{ and } \; B \ll M.
\]
Also, by \eqref{g}, and noting that the sum $\sum_{d = 1}^\infty \frac{g(d)}{d}$ converges absolutely in either case, we have
\[
\sum_{d = 1}^\infty \frac{g(d)}{d} = \prod_{{\begin{substack} {\ell \nmid m_E } \end{substack}}} \left( 1 + \frac{g(\ell)}{\ell} \right) = \prod_{\ell \nmid m_E} \frac{|H_E(\ell)_0| + (\ell-1)| H_E(\ell)_1|}{\ell |H_E(\ell)_1|} = C^{-1}.
\]
Inserting this into \eqref{finalsum1} and using the fact that
\[
\sum_{{\begin{substack} {b \mod m_E \\ b \equiv a \mod q} \end{substack}}} \frac{ |G_E(m_E)_b|}{|G_E(m_E)|} \; = \;  \frac{|G_E(q)_a|}{|G_E(q)|},
\]
we conclude that
\[
\sum_{{\begin{substack} {r \equiv a \mod q \\ A < r \leq A+B} \end{substack}}} C_{E,r} \; = \; \phi_E(0) \cdot B \cdot \frac{|G_E(q)_a|}{|G_E(q)|} \; + \;
\begin{cases}
	O_{E} \left( q \cdot \log^{3} M \right) & \text{ if $E$ has CM} \\
	O_E(q) & \text{ if $E$ has no CM.}
\end{cases}
\]
Proposition \ref{asymprop} now follows at once from the next
\begin{lemma} \label{CMcaselemma}
We have
\begin{equation*} \label{relationwithchebfactor}
\frac{|G_E(q)_a|}{|G_E(q)|} \; = \;
\gl_E \left( \gd_{a,q} - \frac{\gamma(E,a,q)}{2} \right),
\end{equation*}
where
\[
\gl_E :=
\begin{cases}
	2 & \text{ if $E$ has CM} \\
	1 & \text{ if $E$ has no CM.}
\end{cases}
\]
\end{lemma}

\noindent \emph{Proof of Lemma \ref{CMcaselemma}.}
In case $E$ has no CM, the result is immediate.  We turn to the CM case.  Suppose first that $K \subseteq \mbq(E[q])$.  Noting the disjoint union
\[
\gal(\mbq(E[q])/\mbq) = \gal(\mbq(E[q])/K) \sqcup \left( \gal(\mbq(E[q])/\mbq) - \gal(\mbq(E[q])/K) \right),
\]
and that every matrix in $\left( \gal(\mbq(E[q])/\mbq) - \gal(\mbq(E[q])/K) \right)$ has trace zero, Lemma \ref{CMcaselemma} follows in this case.  If $K$ is not contained in $\mbq(E[q])$, then we must have either $q=1$ or $q=2$ (see \cite[Lemma 6]{murty}, for example).  The case $q = 1$ is trivial, and if $q=2$, we see that
\[
\gal(\mbq(E[2])/\mbq) \simeq \gal(K(E[2])/K) \hookrightarrow (\mc{O}/2\mc{O})^*.
\]
By \eqref{OmodellO}, it follows that $\gal(\mbq(E[2])/\mbq)$ is cyclic of order $1$, $2$, or $3$.  We will now argue that the ``cyclic of order $3$'' case never occurs.  To see this, first note that if $\gal(\mbq(E[2])/\mbq)$ is cyclic of order $3$, then the discriminant of $E$ is a perfect square.  Indeed, one may identify $\text{Aut}(E[2])$ with $S_3$, the symmetric group on $3$ letters, by considering its action on the non-identity $2$-torsion points
\[
E[2] - \{ (\infty,\infty) \} = \{ (e_1,0), (e_2,0), (e_3,0) \},
\]
where $E$ is given by the Weierstrass model $y^2 = (x-e_1)(x-e_2)(x-e_3)$.  If $\gal(\mbq(E[2])/\mbq)$ is cyclic of order three, then under this association it must correspond to the alternating group $A_3$.  But then by Galois theory,
\[
\sqrt{\gD_E} = (e_1-e_2)(e_1-e_3)(e_2-e_3) \in \mbq,
\]
and so $\gD_E$ is a perfect square.  Now consider the explicit Weierstrass equations
\[
y^2 = x^3+ax, \quad y^2 = x^3 + b, \quad y^2 = x^3 -3j(j-1728)^3 x + 2j(j-1728)^5
\]
with $j$-invariants $1728$, $0$, and
\[
\begin{split}
j \in \{ &54000, -12288000, 287496, -3375, 16581375, 8000, -32768, -884736, \\
&-884736000, -147197952000, -262537412640768000 \},
\end{split}
\]
respectively.  Any CM elliptic curve over $\mbq$ is $\ol{\mbq}$-isomorphic to one of these models, and except for the curves with $j$-invariant $1728$, the square-free part of the discriminant $\gD = -16(4a^3+27b^2)$ is independent of the model chosen.  One computes the discriminants to be
\[
-2^8 a^3, \quad -2^43^3b^2, \quad 2^{12}3^6 j^2 (j-1728)^9.
\]
The only time any of these is a perfect square is for the curve $y^2 = x^3+ax$, when $a = -t^2$, in which case $E[2]$ is rational.  Thus, $\gal(\mbq(E[2])/\mbq)$ is never cyclic of order $3$, and so must be cyclic of order $1$ or $2$, representable by matrices as
\[
\gal(\mbq(E[2])/\mbq) \quad \simeq \quad \left\{ \begin{pmatrix} 1 & 0 \\ 0 & 1 \end{pmatrix} \right\} \; \text{ or } \; \left\{ \begin{pmatrix} 1 & 1 \\ 0 & 1 \end{pmatrix}, \begin{pmatrix} 1 & 0 \\ 0 & 1 \end{pmatrix} \right\}.
\]
In either case, we have
\[
\frac{| G_E(2)_0 |}{| G_E(2) |} = \gd_{0,2} = 1 \quad \text{ and } \quad \frac{| G_E(2)_1 |}{| G_E(2) |} = \gd_{1,2} = 0,
\]
upon which Lemma \ref{CMcaselemma} follows in this case.
\hfill $\Box$ \\

We have now completed the proof of Proposition \ref{asymprop}. \hfill $\Box$ \\

\section{Consistency with Chebotarev density} \label{consistencywithchebotarevdensity}
We will now verify the consistency of our refinement with the Chebotarev theorem for the $q$-th division field of $E$. More precisely, we establish the following.

\begin{theorem} \label{consistencythm}
Conjecture \ref{refinement} implies the asymptotic \eqref{chebotarev}.
\end{theorem}
\begin{proof}
Let $F_{E,r}(x)$ be defined as in \eqref{fdef}, {\it i.e.}, $F_{E,r}(x)$ is the main term in \eqref{refconj}. 
The statement of the theorem follows from \eqref{deuob} and the asymptotic estimate
\begin{equation*}\label{whatweneed}
\sum_{{\begin{substack} {r \equiv a \mod q \\ 0 < |r| \leq 2\sqrt{x}} \end{substack}}} F_{E,r}(x) \; = \; \left( \gd_{a,q} - \frac{\gamma(E,a,q)}{2} \right)\cdot Li(x) \; + \; O_{E}(q \sqrt{x} \log^{3} x),
\end{equation*}
which we shall prove in the following.  We remark that, in the (more straightforward) non-CM case, one may obtain the stronger error term $O_E(q \sqrt{x}/\log x)$.  The CM case is complicated a bit by the fact that $\phi_E$ has a singularity at the point $1$, which necessitates a truncation parameter $\gd>0$ which will eventually approach zero.  We will prove the CM case, noting that the non-CM case follows in much the same way, but without the parameter $\gd$.

We begin by splitting the left-hand sum as
\[
\sum_{{\begin{substack} {r \equiv a \mod q \\ 0 < |r| \leq 2\sqrt{x}} \end{substack}}} F_{E,r}(x) = \sum_{{\begin{substack} {r \equiv a \mod q \\ 2 < r \leq 2\sqrt{x}} \end{substack}}} F_{E,r}(x) + \sum_{{\begin{substack} {r \equiv a \mod q \\ -2\sqrt{x} < r \leq -3} \end{substack}}} F_{E,r}(x) + O\left( \frac{\sqrt{x}}{\log x} \right).
\]
We will now show that
\begin{equation} \label{finalgoal}
\sum_{{\begin{substack} {r \equiv a \mod q \\ 2 < r \leq 2\sqrt{x}} \end{substack}}} F_{E,r}(x) = \frac{1}{2} \cdot \left( \gd_{a,q} - \frac{\gamma(E,a,q)}{2} \right)\cdot Li(x) \; + \; O_{E}(q \sqrt{x} \log^{3} x),
\end{equation}
the proof that
\[
\sum_{{\begin{substack} {r \equiv a \mod q \\ -2\sqrt{x} < r \leq -3} \end{substack}}} F_{E,r}(x) = \frac{1}{2} \cdot \left( \gd_{a,q} - \frac{\gamma(E,a,q)}{2} \right)\cdot Li(x) \; + \;
O_{E}(q \sqrt{x} \log^{3} x)
\]
being essentially the same. Remembering that $\Phi_E(z)=\phi_E(z)/\phi(0)$, the left-hand side of \eqref{finalgoal} is the limit as $\gd \rightarrow 0^+$ of
\[
\frac{1}{\phi_E(0)} \cdot \sum_{{\begin{substack} {r \equiv a \mod q \\ 2 < r \leq 2\sqrt{x}} \end{substack}}} C_{E,r} \int_{r^2/4 + \gd}^{x + \gd} \frac{\phi_E(r/(2\sqrt{t}))}{2\sqrt{t} \log t} dt.
\]
By partial summation and by integration by parts, the above expression is equal to
\[
- \frac{1}{\phi_E(0)} \cdot \int_{3}^{2\sqrt{x}} \left( \sum_{{\begin{substack} {r \equiv a \mod q \\ 2 < r \leq y} \end{substack}}} C_{E,r} \right) \cdot \frac{d}{dy} \left( \int_{y^2/4+\gd}^{x+\gd} \frac{\phi_E(y/(2\sqrt{t}))}{2\sqrt{t} \log t} dt  \right) dy.
\]
We now invoke the estimate \eqref{asymptotic}, obtaining
\begin{equation} \label{justusedprop3}
\begin{split}
- \gl_E \cdot  \left(\gd_{a,q}-\frac{\gamma(E,a,q)}{2}\right) \cdot & \int_{3}^{2\sqrt{x}} (y-3) \cdot \frac{d}{dy} \left( \int_{y^2/4+\gd}^{x+\gd} \frac{\phi_E(y/(2\sqrt{t}))}{2\sqrt{t} \log t} dt \right) dy \\
+ &O_{E}(q \sqrt{x + \gd} \log^{3} x),
\end{split}
\end{equation}
where
\[
\gl_E :=
\begin{cases}
	2 & \text{ if $E$ has CM} \\
	1 & \text{ if $E$ has no CM}
\end{cases}
\]
and for the error bound, we have used the fact that
\[
\int_{9/4+\gd}^{x+\gd} \frac{\phi_E(3/(2\sqrt{t}))}{2\sqrt{t} \log t} dt \; \ll \; \sqrt{x+\gd}.
\]
Integrating by parts, we see that the integral in the main term is then equal to
\[
\begin{split}
\int_3^{2\sqrt{x}} \int_{y^2/4+\gd}^{x+\gd} \frac{\phi_E(y/(2\sqrt{t}))}{2\sqrt{t} \log t} dt dy \; =& \; \int_{9/4 + \gd}^{x+\gd} \frac{1}{2\sqrt{t} \log t} \int_3^{2\sqrt{t-\gd}}  \phi_E(y/(2\sqrt{t}))\; dy dt \\
=& \; \int_{9/4+\gd}^{x+\gd} \frac{1}{\log t} \int_{3/(2\sqrt{t})}^{2\sqrt{t-\gd}/(2\sqrt{t})} \phi_E(z) \; dz dt  \\
=& \; \int_{9/4+\gd}^{x+\gd} \frac{dt}{\log t} \cdot \int_0^1 \phi_E(z)dz  \; + \;
	O \left( \frac{\sqrt{x+\gd}}{\log x} \right),
\end{split}
\]
where we have made use of the facts that, for (say) $0 \leq \gl \leq 1/2$,
\[
2\pi \cdot \int_0^\gl \phi_E(t) dt = \arcsin(\gl) = \gl + O(\gl^3)
\]
and
\[
\arcsin(1) - \arcsin(1-\gl) = O(\sqrt{\gl}),
\]
which imply that
\[
\int_{9/4+\gd}^{x+\gd} \frac{1}{\log t} \int_{0}^{3/(2\sqrt{t})} \phi_E(z) dz dt \; = \; O\left(\frac{\sqrt{x+\gd}}{\log x}\right)
\]
and
\[
\int_{9/4+\gd}^{x+\gd} \frac{1}{\log t} \int_{1-\gd/t}^1 \phi_E(z) dz dt \; = \; O\left( \frac{\sqrt{\gd} \sqrt{x+\gd}}{\log x} \right).
\]
Inserting this into \eqref{justusedprop3}, using that
\[
\int_0^1 \phi_E(z)dz =
\begin{cases}
	1/4 & \text{ if $E$ has CM} \\
	1/2 & \text{ if $E$ has no CM}
\end{cases}
= \frac{1}{2 \gl_E},
\]
and letting $\gd \rightarrow 0^+$, the asymptotic estimate \eqref{finalgoal} and hence Theorem \ref{consistencythm} is proved.
\end{proof}

\section{Consistency with Sato-Tate} \label{consistencywithsatotate}
In this section we will establish that Conjecture \ref{refinement} implies the Sato-Tate conjecture. We deduce this from the following stronger result which implies that the correcting factor $\Phi_E(z)$ in the main term in \eqref{refconj} is the only possibility, {\it i.e.}, it {\it must} be of the form given in \eqref{phidef2}.

\begin{theorem} \label{satotateth} Let $E$ be an elliptic curve over $\mbq$, $\phi_E(z)$ be defined by \eqref{phidef} and $C>1$ be any constant. Assume that there exists a continuously differentiable function $\Phi: (-1,1) \rightarrow \mbr$ such that, uniformly for $|r| \leq 2\sqrt{x}$ (excluding $r = 0$ if $E$ has CM),
\begin{equation} \label{assumpnoncm}
\pi_{E,r}(x) \; = \;  C_{E,r}\int_{\max\{2,r^2/4\}}^x \frac{\Phi(r/(2\sqrt{t}))}{2\sqrt{t} \log t} dt + O\left(\frac{\sqrt{x}}{(\log x)^C}\right).
\end{equation}
Then the Sato-Tate conjecture (resp. \eqref{cmdistr} if $E$ has CM) holds if and only if $\Phi(z)=\Phi_E(z)$ for all $z\in (-1,1)$, where $\Phi_E(z)$ is defined as in \eqref{phidef2}.
\end{theorem}

\begin{proof}
By continuity of $\Phi$, we have $\Phi(z)=\Phi_E(z)$ for all $z\in (-1,1)$ if and only if
\begin{equation} \label{equalints}
\int\limits_{\alpha}^{\beta} \Phi(z)dz=\int\limits_{\alpha}^{\beta} \Phi_E(z)dz
\end{equation}
for all $\alpha,\beta$ with $-1< \alpha <\beta< 1$ and $0\not\in [\alpha,\beta]$. Moreover, the equation \eqref{equalints} is equivalent with
$$
\int\limits_{\alpha}^{\beta} \phi(z)dz=\int\limits_{\alpha}^{\beta} \phi_E(z)dz,
$$
where we set
$$
\phi(z):=\phi_E(0)\Phi(z),
$$
and $\phi_E(z)$ is defined as in \eqref{phidef}.
Therefore, to establish the equivalence claimed in the Theorem, it suffices to prove that if \eqref{assumpnoncm} holds, then
\begin{equation} \label{suff}
\sum_{{\begin{substack} { p \leq x \\ \ga \leq \frac{a_E(p)}{2\sqrt{p}} < \gb} \end{substack}}} 1 \sim Li(x)\cdot \int_{\ga}^{\gb}\phi(z)dz \quad \mbox{ as } x\rightarrow\infty
\end{equation}
for all fixed $\alpha,\beta$ satisfying $-1< \alpha <\beta< 1$ and $0\not\in [\alpha,\beta]$.
In the sequel, we assume that $0<\alpha<\beta< 1$. In the complementary case
$-1< \alpha<\beta<0$, \eqref{suff} can be proved similarly.

We note that, for $a_E(p) > 0$ one has
\[
\ga \leq \frac{a_E(p)}{2\sqrt{p}} < \gb \quad \Longleftrightarrow \quad \frac{a_E(p)^2}{4\gb^2} < p \leq \frac{a_E(p)^2}{4\ga^2}.
\]
Thus,
\begin{equation} \label{splitting}
\begin{split}
\sum_{{\begin{substack} { p \leq x \\ \ga \leq \frac{a_E(p)}{2\sqrt{p}} < \gb} \end{substack}}} 1 \; = \; &\sum_{{\begin{substack} {0 < r \leq 2\sqrt{x}\ga} \end{substack}}} \left( \pi_{E,r}\left( \frac{r^2}{4\ga^2} \right) - \pi_{E,r}\left( \frac{r^2}{4\gb^2} \right) \right) \\
+ \; &\sum_{{\begin{substack} {2\sqrt{x}\ga < r \leq 2\sqrt{x}\gb} \end{substack}}} \left( \pi_{E,r}\left( x \right) - \pi_{E,r}\left( \frac{r^2}{4\gb^2} \right) \right).
\end{split}
\end{equation}
We observe that \eqref{suff} follows from \eqref{assumpnoncm}, \eqref{splitting} and the asymptotic estimate
\begin{equation} \label{formerprop}
\begin{split}
&  \sum_{{\begin{substack} { 2 < r \leq 2\sqrt{x}\ga } \end{substack}}} \frac{C_{E,r}}{\phi_E(0)} \int_{r^2/4\gb^2}^{r^2/4\ga^2} \frac{\phi(r/(2\sqrt{t}))}{2\sqrt{t}\log t}dt + \sum_{{\begin{substack} { 2\sqrt{x}\ga < r \leq 2\sqrt{x}\gb } \end{substack}}} \frac{C_{E,r}}{\phi_E(0)} \int_{r^2/4\gb^2}^{x} \frac{\phi(r/(2\sqrt{t}))}{2\sqrt{t}\log t}dt \\
&\quad\quad\quad\quad\quad\quad\quad = Li(x) \cdot \int_{\ga}^{\gb}\phi(z)dz  \; + \;
	O_{\ga, \gb} \left(\sqrt{x} \log^{2} x \right),
\end{split}
\end{equation}
which we shall prove in the following.  Reversing the order of summation and integration, the left-hand side of \eqref{formerprop} becomes
\begin{equation} \label{red}
\int\limits_2^x \frac{1}{2\sqrt{t}\log t} \left(\sum\limits_{2\alpha\sqrt{t}<r\le 2\beta\sqrt{t}} \frac{C_{E,r}}{\phi_E(0)} \cdot \phi\left(\frac{r}{2\sqrt{t}}\right)\right)dt\ +O_{\alpha,\beta}(1).
\end{equation}
We note that by $\delta_{1,0}=1$ and the definition of $\phi_E(z)$ in \eqref{phidef}, the main term on the right-hand side of \eqref{asymptotic} coincides with $\phi(0)\cdot B$ if $q=1$ and $a=0$. Now using
partial summation, Proposition \ref{asymprop} with $q=1,a=0$, and integration by parts, we have
\begin{eqnarray*} \label{parsu}
& & \sum\limits_{2\alpha\sqrt{t}<r\le 2\beta\sqrt{t}} C_{E,r}\cdot \phi\left(\frac{r}{2\sqrt{t}}\right)\\ &=&
\phi(\beta)\sum\limits_{2\alpha\sqrt{t}<r\le 2\beta\sqrt{t}} C_{E,r} -
\int\limits_{2\alpha\sqrt{t}}^{2\beta\sqrt{t}}
\left(\sum\limits_{2\alpha\sqrt{t}<r\le y}C_{E,r}\right) \frac{d}{dy} \phi\left(\frac{y}{2\sqrt{t}}\right)\ dy\\ &=&
\phi_E(0)\cdot (2\beta\sqrt{t}-2\alpha\sqrt{t})\phi(\beta) -
\phi_E(0)\int\limits_{2\alpha\sqrt{t}}^{2\beta\sqrt{t}}
(y-2\alpha\sqrt{t}) \frac{d}{dy} \phi\left(\frac{y}{2\sqrt{t}}\right)\ dy  +
	O_{\ga, \gb} \left(\log^3 t\right) \\
&=& \phi_E(0) \cdot \int\limits_{2\alpha\sqrt{t}}^{2\beta\sqrt{t}} \phi\left(\frac{y}{2\sqrt{t}}\right)dy  \; + \;
	O_{\ga, \gb} \left(\log^3 t\right)
\end{eqnarray*}
where for the estimation of the error term, we have used that the derivative of $\phi$ is continuous and hence bounded on $[\alpha,\beta]$. Thus, \eqref{red} equals
$$
\int\limits_2^x \frac{1}{2\sqrt{t}\log t} \int\limits_{2\alpha\sqrt{t}}^{2\beta\sqrt{t}} \phi\left(\frac{y}{2\sqrt{t}}\right)dydt   \; + \;
	O_{\ga, \gb} \left(\sqrt{x} \log^{2} x \right)
$$
Making the change of variables $y/(2\sqrt{t}) \rightarrow z$, the main term above becomes
$$
Li(x)\cdot \int_{\ga}^{\gb}\phi(z)dz,
$$
which proves \eqref{formerprop} and hence Theorem \ref{satotateth}.
\end{proof}

\section{Numerical evidence} \label{numericalevidence}

We conclude with some supporting numerical evidence.  The five figures below display data for the single elliptic curve $E$ given by the Weierstrass equation
\[
Y^2 = X^3 + 6X - 2.
\]

In Figure \ref{fig:fm1}, we plot the function $v := \pi_{E,r}(4\cdot 10^7)$ as a function of the variable $r$.  In Figure \ref{fig:f0}, we plot our approximation $v := F_{E,r}(4\cdot 10^7)$ as a function of $r$.  This elliptic curve has $m_E = 6$, and the ``main factor''
\[
\frac{m_E \cdot |\gal(\mbq(E[m_E])/\mbq)_r|}{|\gal(\mbq(E[m_E])/\mbq)|}
\]
of the constant $C_{E,r}$ takes on $4$ distinct values $\{ 1/2, 3/4, 9/8, 7/4 \}$ as $r$ ranges over the integers, which accounts for the $4$ distinct bands visible in Figures \ref{fig:fm1} and \ref{fig:f0}.

\begin{figure}
\centering
\scalebox{0.38}
{\includegraphics{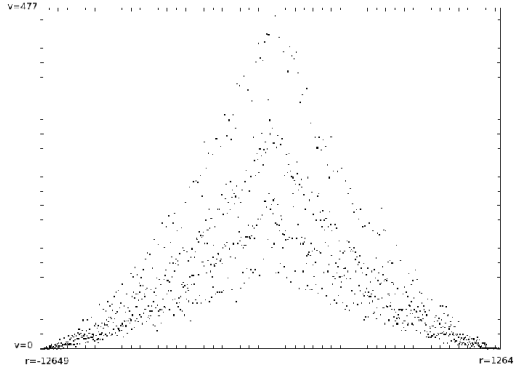}}
\caption{The function $v = \pi_{E,r}(4 \cdot 10^7)$, as a function of $r$}
\label{fig:fm1}
\end{figure}

\begin{figure}
\centering
\scalebox{0.38}
{\includegraphics{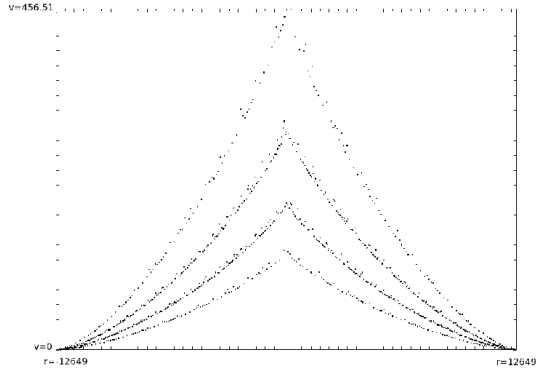}}
\caption{The approximation $v = F_{E,r}(4 \cdot 10^7)$, as a function of $r$}
\label{fig:f0}
\end{figure}

We then plot various forms of the error in the approximation.  In Figure \ref{fig:f1}, we plot the absolute error
\[
v = \pi_{E,r}(4 \cdot 10^7) - F_{E,r}(4 \cdot 10^7),
\]
while in Figure \ref{fig:f2}, we plot the relative error
\[
v = \frac{ \pi_{E,r}(4 \cdot 10^7) - F_{E,r}(4 \cdot 10^7)}{F_{E,r}(4\cdot 10^7)}.
\]

\begin{figure}
\centering
\scalebox{0.38}
{\includegraphics{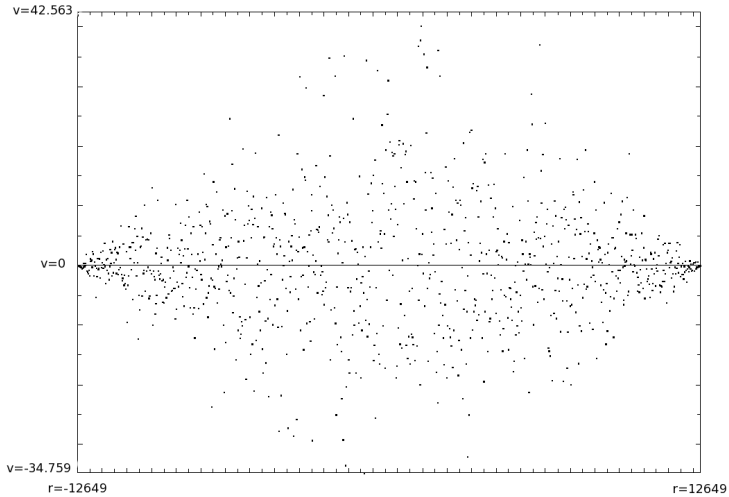}}
\caption{The absolute error $v = \pi_{E,r}(4 \cdot 10^7) - F_{E,r}(4\cdot 10^7)$}
\label{fig:f1}
\end{figure}

\begin{figure}
\centering
\scalebox{0.38}
{\includegraphics{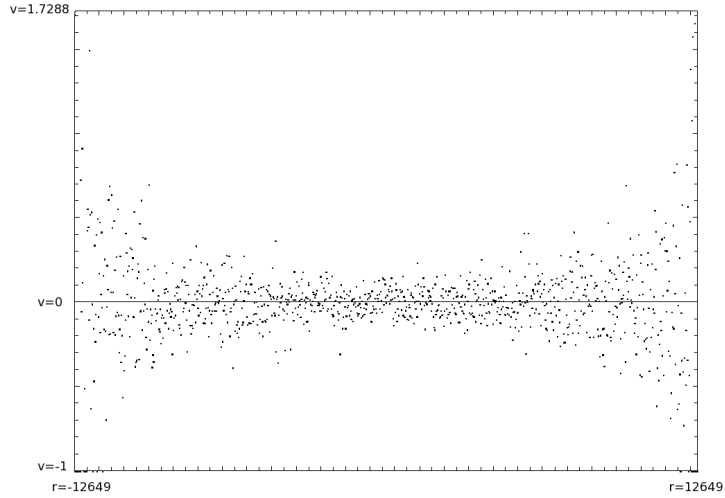}}
\caption{The relative error $v = \frac{ \pi_{E,r}(4 \cdot 10^7) - F_{E,r}(4\cdot 10^7)}{F_{E,r}(4\cdot 10^7)}$}
\label{fig:f2}
\end{figure}

Note that the absolute (resp. relative) error is significantly smaller (resp. larger) at the ends of the graph than in the middle.  This comes from the fact that we are approximating an integer valued function with a continuous one.

We remark that in practice, the main difficulty in obtaining numerical data on these error terms lies in the constants $C_{E,r}$, which are difficult to compute in general.  However, the elliptic curve we are considering is a Serre curve (see \cite[p. 318]{serre} and also \cite[p. $51$]{langtrotter}), so we may use Proposition $11$ of \cite{jones}, which computes $C_{E,r}$ explicitly for any Serre curve.

In Figure \ref{fig:f3} we plot the error relative to square root of the main term, which looks remarkably like random noise.

\begin{figure}
\centering
\scalebox{0.38}
{\includegraphics{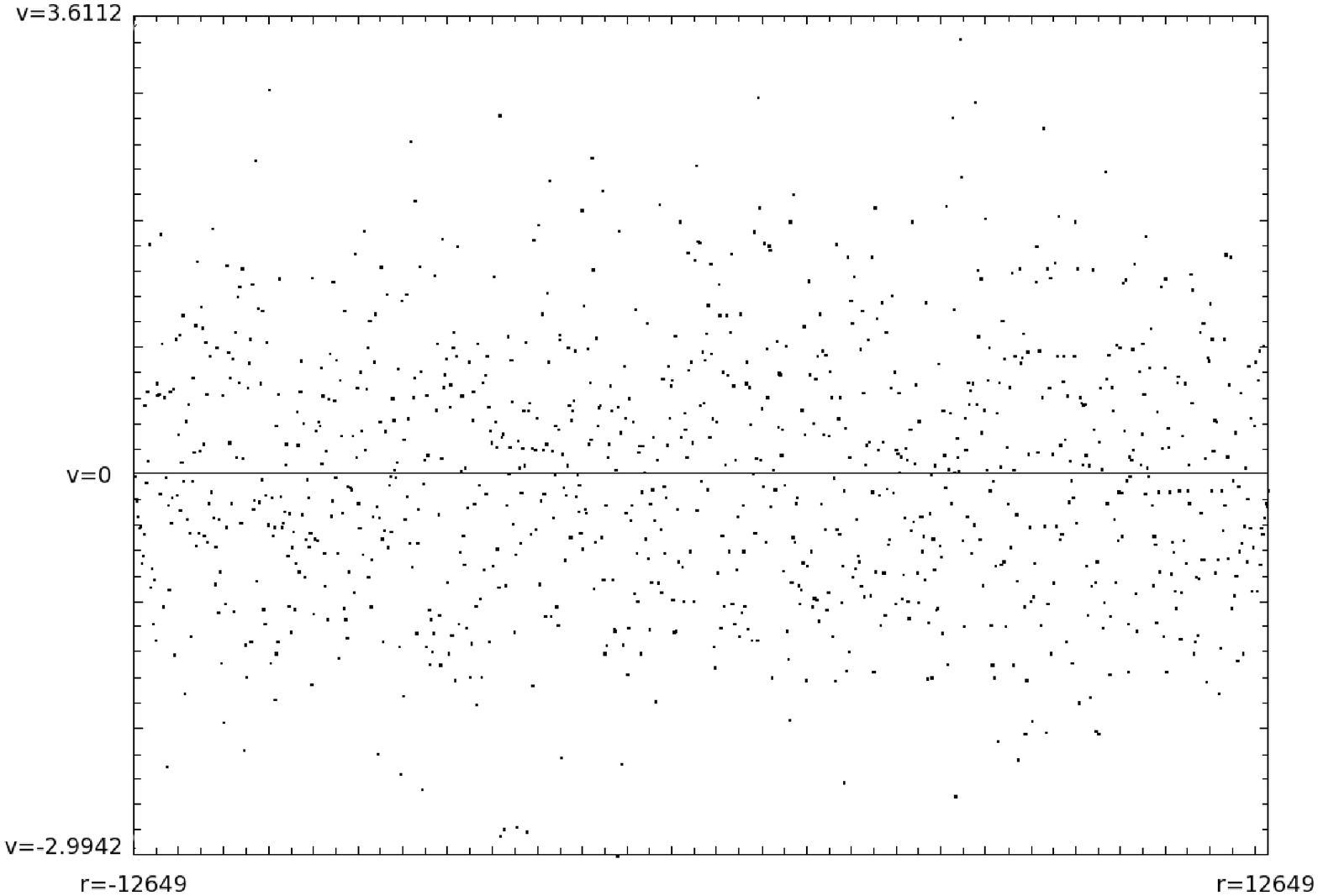}}
\caption{$v = \frac{\text{Error}}{\sqrt{\text{main term}}}
\; = \; \frac{ \pi_{E,r}(4 \cdot 10^7) - F_{E,r}(4\cdot 10^7) }{\sqrt{F_{E,r}(4\cdot 10^7)}}$}
\label{fig:f3}
\end{figure}

Finally, in Figures \ref{fig:cmfm1} -- \ref{fig:cmf3} we plot the corresponding data for the elliptic curve $E$ given by the Weierstrass equation
\[
Y^2 = X^3 - 768108000X + 8194304162000,
\]
which has CM by the complex order of discriminant $-27$ (i.e. by the unique order of index $3$ in $\mbz[1/2 + \sqrt{-3}/2]$).

\begin{figure}
\centering
\scalebox{0.38}
{\includegraphics{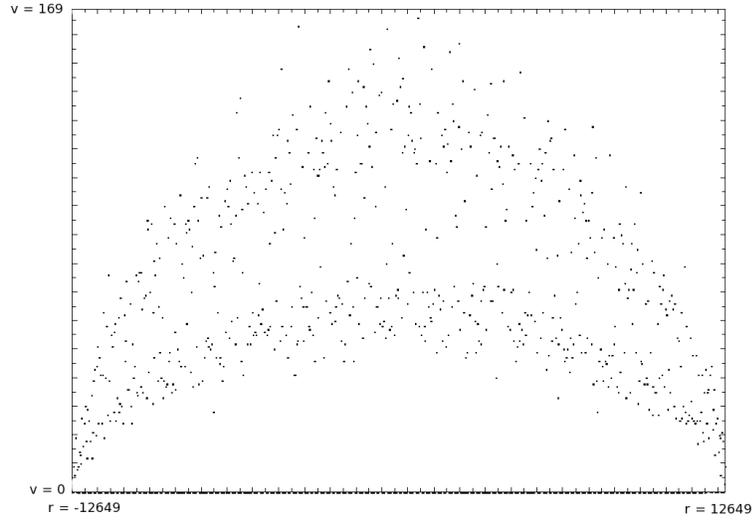}}
\caption{The function $v = \pi_{E,r}(4 \cdot 10^7)$, as a function of $r$}
\label{fig:cmfm1}
\end{figure}

\begin{figure}
\centering
\scalebox{0.38}
{\includegraphics{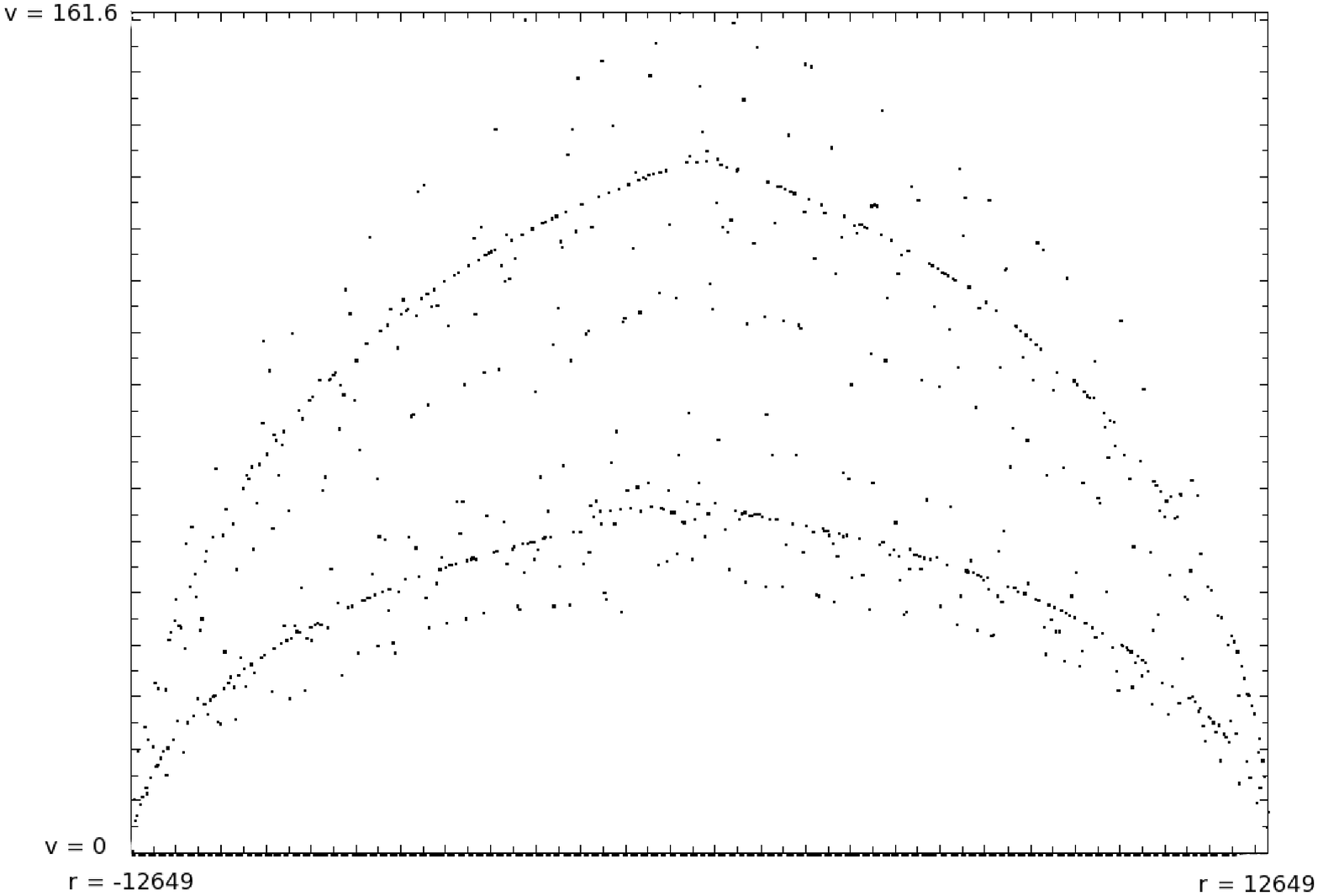}}
\caption{The approximation $v = F_{E,r}(4\cdot 10^7)$, as a function of $r$}
\label{fig:cmf0}
\end{figure}

\begin{figure}
\centering
\scalebox{0.38}
{\includegraphics{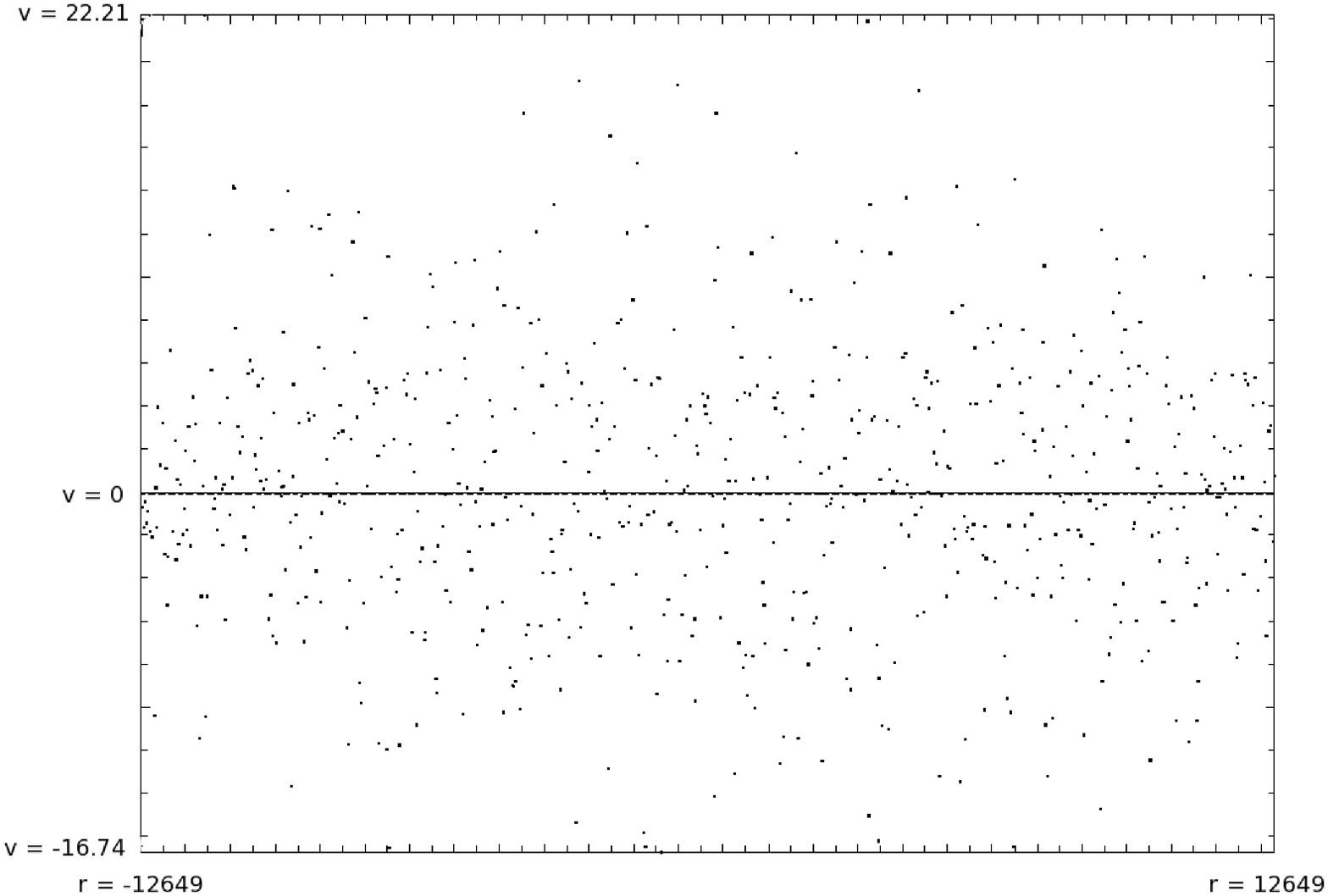}}
\caption{The absolute error $v = \pi_{E,r}(4 \cdot 10^7) - F_{E,r}(4\cdot 10^7)$}
\label{fig:cmf1}
\end{figure}

\begin{figure}
\centering
\scalebox{0.38}
{\includegraphics{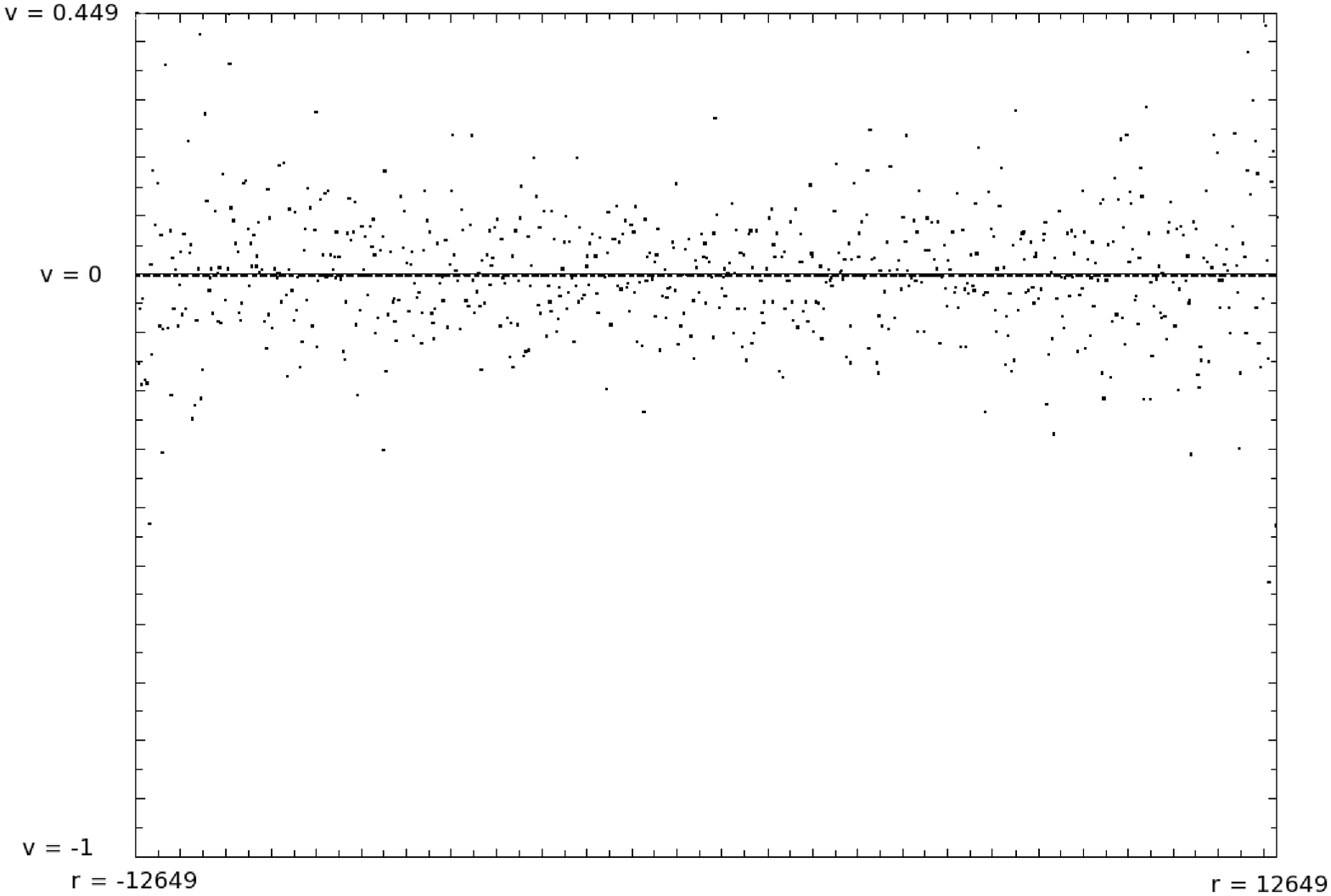}}
\caption{The relative error $v = \frac{ \pi_{E,r}(4 \cdot 10^7) - F_{E,r}(4 \cdot 10^7)}{F_{E,r}(4\cdot 10^7)}$}
\label{fig:cmf2}
\end{figure}

\begin{figure}
\centering
\scalebox{0.38}
{\includegraphics{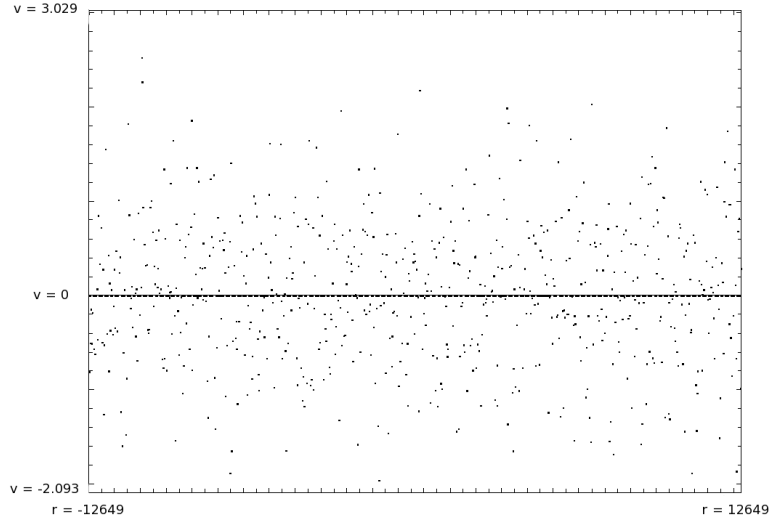}}
\caption{$v = \frac{\text{Error}}{\sqrt{\text{main term}}}
\; = \; \frac{ \pi_{E,r}(4 \cdot 10^7) - F_{E,r}(4 \cdot 10^7)}{ \sqrt{F_{E,r}(4 \cdot 10^7)}}$}
\label{fig:cmf3}
\end{figure}


\begin{thebibliography}{99}

\bibitem{baierzhao} S. Baier and L. Zhao, \emph{The Sato-Tate conjecture on average for small angles}, to appear in Trans. Amer. Math. Soc.
\bibitem{davidpappalardi} C. David and F. Pappalardi, \emph{Average Frobenius
Distributions of Elliptic Curves}, Int. Math. Res. Not. \textbf{4} (1999), 165--183.
\bibitem{deur} M. Deuring, {\it Die Typen der Multiplikatorenringe
elliptischer Funktionenk\"orper}, Abh. Math. Sem. Hansischen Univ. 14 (1941)
197-272.
\bibitem{jones} N. Jones, \emph{Averages of elliptic curve constants}, preprint.
\bibitem{jonestc} N. Jones, \emph{A bound for the ``torsion conductor'' of a non-CM elliptic curve}, to appear in Proceedings of the AMS.
\bibitem{langtrotter} S. Lang and H. Trotter, \emph{Frobenius distributions in $GL_2$-extensions}, Lecture notes in Math., 504, Springer-Verlag, Berlin, 1976.
\bibitem{murty} M. Ram Murty, \emph{On Artin's Conjecture}, Journal of Number Theory, \textbf{16} (1983) 147--168.
\bibitem{serre} J.-P. Serre, \emph{Propri\'{e}t\'{e}s galoisiennes des points d'ordre fini des courbes elliptiques}, Invent. Math., \textbf{15} (1972), 259 - 331.
\bibitem{silverman2} J. Silverman, \emph{Advanced topics in the Arithmetic of Elliptic curves}, Springer-Verlag, New York, 1994.
\bibitem{Tate} J.T. Tate, {\it
Algebraic cycles and poles of zeta functions},
Arithmetical algebraic Geom., Harper and Row, New York, 1965.
\bibitem{Tayl} R. Taylor, {\it Automorphy for some $l$-adic lifts of automorphic mod $l$ representations II}, preprint, available at www.math.harvard.edu/$\sim$rtaylor.
\end{thebibliography}
\end{document}